\documentclass{amsart}%
\usepackage{amsfonts}
\usepackage{graphicx}
\usepackage{amsmath}
\usepackage{amssymb}
\usepackage{rotating}%
\newtheorem{theorem}{Theorem}

\newtheorem{definition}[theorem]{Definition}

\newtheorem{lemma}[theorem]{Lemma}

\newtheorem{proposition}[theorem]{Proposition}
\newtheorem{remark}[theorem]{Remark}

\begin{document}
\title[ ]{A $q$-Umbral Approach to $q$-Appell Polynomials}
\author{Marzieh Eini Keleshteri and Naz{i}m I. Mahmudov}
\address{Mathematics Department, Eastern Mediterranean University, Famagusta, North
Cyprus, via Mersin 10, Turkey}
\email{marzieh.eini@emu.edu.tr, nazim.mahmudov@emu.edu.tr}
\subjclass{Special Functions}
\keywords{Umbral, $q$-polynomials, Appell, Euler, Bernoulli, Genocchi}

\begin{abstract}
In this paper we aim to specify some characteristics of the so called family of $q$-Appell Polynomials by using $q$-Umbral calculus. Next in our study, we focus on $q$-Genocchi numbers and polynomials as a famous member of this family. To do this, firstly we show that any arbitrary polynomial can be written based on a linear combination of $q$-Genocchi polynomials. Finally, we approach to the point that similar properties can be found for the other members of the class of $q$-Appell polynomials.

\end{abstract}
\maketitle

\section{Introduction and preliminaries}

\subsection{$q$-Calculus}
Throughout this work we consider the notation $\mathbb{N}$ as the set of natural numbers, $\mathbb{N}_{0}$ as the set of positive integers and
$\mathbb{C}$ as the set of complex numbers. We refer the readers to
\cite{Andrews} for all the following $q$-standard notations. The $q$-shifted
factorial is defined as%
\[
(a;q)_{0}=1,\quad(a;q)_{n}=\prod\limits_{j=0}^{n-1}(1-q^{j}a),\quad
n\in\mathbb{N},(a;q)_{\infty}=\prod\limits_{j=0}^{\infty}(1-q^{j}%
a),\quad|q|<1,a\in\mathbb{C}.
\]
The $q$-numbers and $q$-factorial are defined by%
\[
\lbrack a]_{q}=\frac{1-q^{a}}{1-q}\quad(q\neq1);\quad\lbrack0]!=1;\quad\lbrack
n]_{q}!=[1]_{q}[2]_{q}\ldots\lbrack n]_{q},\quad\lbrack2n]_{q}!!=[2n]_{q}%
[2n-2]_{q...}[2]_{q},\text{ \ \ }n\in\mathbb{N},a\in\mathbb{C},
\]
respectively. The $q$-polynomial coefficient is defined by%
\[
\left[
\begin{array}
[c]{c}%
n\\
k
\end{array}
\right]  _{q}=\dfrac{[n]_{q}!}{[k]_{q}![n-k]_{q}!}.
\]
The $q$-analogue of the function $(x+y)^{n}$ is defined by%
\begin{equation}
(x+y)_{q}^{n}:=\sum\limits_{k=0}^{n}\left[
\begin{array}
[c]{c}%
n\\
k
\end{array}
\right]  _{q}q^{1/2k(k-1)}x^{n-k}y^{k},\quad n\in\mathbb{N}_{0}. \label{00}%
\end{equation}
The $q$-binomial formula is known as%
\begin{equation}
(1-a)_{q}^{n}=\prod\limits_{j=0}^{n-1}(1-q^{j}a)=\sum\limits_{k=0}^{n}\left[
\begin{array}
[c]{c}%
n\\
k
\end{array}
\right]  _{q}q^{1/2k(k-1)}(-1)^{k}a^{k}.\label{Binomial}
\end{equation}
In the standard approach to the $q$-calculus, one of the $q$-analogues of the exponential function is defined as%
\begin{equation}
e_{q}\left(  z\right) =\sum_{n=0}^{\infty}\frac{z^{n}}{\left[  n\right]
_{q}!}=\prod_{k=0}^{\infty}\frac{1}{\left(  1-\left(  1-q\right)
q^{k}z\right)  },\ \ \ 0<\left\vert q\right\vert <1,\ \left\vert z\right\vert
<\frac{1}{\left\vert 1-q\right\vert }, z\in\mathbb{C} \label{exp}.
\end{equation}

The $q$-derivative of a function $f$ at point $0\neq z\in\mathbb{C}$, is
defined as%
\begin{equation}
D_{q}f\left(  z\right)  :=\frac{f\left(  qz\right)  -f\left(  z\right)
}{qz-z},\ \ \ \ 0<\left\vert q\right\vert <1. \label{q-der}%
\end{equation}
From this we easily see that
\[
D_{q}e_{q}\left(  z\right)  =e_{q}\left(  z\right).
\]
Moreover, Jackson definite integral of an arbitrary function $f(x)$ is defined as
\cite{Kac}%

\begin{equation}
\int_{0}^{x} f(x)d_{q}x=(1-q)\sum\limits_{n=0}^{\infty}xq^{j}f(xq^{j}),\text{
\ \ }0<q<1. \label{int}%
\end{equation}
Noting to the definitions of $q$-derivative and $q$-integral of a function $f(x)$ in (\ref{q-der}) and (\ref{int}), it is clear that
\begin{equation}
D_{q}\int_{0}^{x} f(x)d_{q}x=f(x), \quad \int_{a}^{b} f(x)d_{q}x=\int_{0}^{b} f(x)d_{q}x-\int_{0}^{a} f(x)d_{q}x.
\end{equation}
According to Carlitz's extension of the classical Bernoulli and Euler polynomials, \cite{calitz1}, \cite{calitz2}, \cite{calitz3}, $q$-Bernoulli and $q$-Euler polynomials are defined by means of the following generating functions
\begin{equation}\label{Bernoulli}
\frac{t}{e_{q}\left(  t\right)-1}e_{q}\left(  tx\right)=\sum_{n=0}^{\infty}B_{n,q}(x)\frac{z^{n}}{\left[  n\right]
_{q}!},\end{equation}
\begin{equation}\label{Euler}
\frac{2}{e_{q}\left(  t\right)+1}e_{q}\left(  tx\right)=\sum_{n=0}^{\infty}E_{n,q}(x)\frac{z^{n}}{\left[  n\right]
_{q}!},\\
\end{equation}
respectively.
In a similar way, according to Kim, $q$-Genocchi polynomials can be defined by means of the following generating function, \cite{Kim}
\begin{equation}\label{Genocchi}
\frac{2t}{e_{q}\left(  t\right)+1}e_{q}\left(  tx\right)=\sum_{n=0}^{\infty}G_{n,q}(x)\frac{z^{n}}{\left[  n\right]
_{q}!}.\\
\end{equation}
For $x=0,\quad B_{n,q}(0)=B_{n,q},\quad E_{n,q}(0)=E_{n,q},$ and $ G_{n,q}(0)=G_{n,q}$, are called the $n$-th $q$-Bernoulli, $q$-Euler, and $q$-Genocchi numbers, respectively. \\
The research on the above mentioned polynomials is vast. The interested readers are referred to \cite{Araci}-\cite{Kurt2} to see various extensions and relations regarding these numbers and polynomials.\\
The class of Appell polynomials for the first time attracted Appell's note in 1880,
\cite{Appell}. In his studies, Appell, characterized this family of polynomials completely. Later, the research done by Throne \cite{Throne}, Sheffer \cite{Amer}, and Varma \cite{Varma} from different points of views, developed the aforementioned class of polynomials. Sheffer, also, showed that how the properties of Appell polynomials hold well for his generalization. In 1954, Sharma and Chak, for the first time, introduced a q-analogue for the family of $q$-Appell polynomials and called this sequence of polynomials as $q$-Harmonic, \cite{Shar}. In the light of their works, Al-Salam, in 1967, reintroduced the family of $q$-Appell polynomials $\{A_{n,q}(x)\}_{n=0}^{\infty}$, and studied some of its properties, \cite{Alsalam}. According to his definition, the n-degree polynomials $A_{n,q}(x)$, are called $q$-Appell provided that any $A_{n,q}(x)$ holds the following $q$-differential equation%

\begin{equation}
D_{q,x}(A_{n,q}(x))=[n]_{q}A_{n-1,q}(x),\text{ \ \ for }n=0,1,2,...\text{ .}
\label{1}%
\end{equation}
This is equivalent to define this family of polynomials by means of the following generating function
$A_{q}(t)$, as follows%

\begin{equation}
A_{q}(x,t):=A_{q}(t)e_{q}(tx)=\sum_{n=0}^{\infty}A_{n,q}(x)\frac{t^{n}%
}{\left[  n\right]  _{q}!},\text{ \ \ }0<q<1, \label{2}%
\end{equation}

where
\begin{equation}
A_{q}(t):=\sum_{n=0}^{\infty}A_{n,q}\frac{t^{n}}{\left[  n\right]  _{q}%
!},\ \ A_{q}(t)\neq0, \label{3}%
\end{equation}

is an analytic function at $t=0$, and $A_{n,q}(x):=A_{n,q}(0).$ The formal power series $A_{q}(t)$, in the above definition, is called the determining function of the class of $q$-Appell polynomials $\{A_{n,q}(x)\}$.\\
In his researches, algebraically, Al-Salam showed that the class of all $q$-Appell polynomials is a maximal commutative subgroup of the group of all polynomial sets.
Later, in 1982, Srivastava's studies on the family of $q$-Appell polynomials led to more characterization and clarification of these type of polynomials, \cite{Sri}. During the past few decades, the class of $q$-Appell polynomials has been studied from different aspects, \cite{Kishan}, \cite{Ernst}, and using different thechniques, \cite{Lou}. Recently, the $q$-difference equations as well as recurrence relations satisfied by sequence of $q$-Appell polynomials are derived by Mahmudov, \cite{Mah4}.
\subsection{$q$-Umbral Calculus}
In 1978, Roman and Rota viewed the classical umbral calculus from a new perspective and proposed an interesting approach based on a simple but innovative indication for the effect of linear functionals on polynomials, which Roman later called it the modern classical umbral calculus, \cite{RomanRota}, \cite{Roman}. Using this new umbral calculus, they defined the sequence of Sheffer polynomials whose their characteristics proved that this new proposed family of polynomials is equivalent to the family of polynomials of type zero which was previously introduced by Sheffer, \cite{IM}. Roman, also, proposed a similar umbral approach under the area of nonclassical umbral calculus which is called $q$-umbral calculus, \cite{Roman}, \cite{Roman1}, \cite{Roman2}. Inspired by his work, in the following, we recast the results of $q$-umbral calculus for $q$-Appell polynomials.\\
Let $\mathbb{C}$ be the field of complex numbers and  $\mathcal{F}$ set of all formal power $q$-series in the variable t over $\mathbb{C}$. In other words, $f(t)$ is an element of $\mathcal{F}$ if
\begin{equation}
f(t)=\sum\limits_{k=0}^{\infty}\dfrac{a_k}{[k]_q!}t^k, \label{fs}
\end{equation}
where $a_k$ is in $\mathbb{C}$.\\
Let $\mathcal{P}$ be the algebra of all polynomials in variable $x$ over $\mathbb{C}$. Let $\mathcal{P}^*$ be the vector space of all linear functionals on $\mathcal{P}$. The action of a linear functional $L$ on an arbitrary polynomial $p(x)$ is denoted by $ \langle L | p(x)\rangle$. We remind that the vector space addition and scalar multiplication operations on $\mathcal{P}^*$ are defined by $\langle L+M | p(x)\rangle=\langle L | p(x)\rangle+\langle M | p(x)\rangle,$ and $\langle cL | p(x)\rangle=c\langle L | p(x)\rangle,$ for any constant $c \in \mathbb{C}.$\\
The formal power $q$-series in (\ref{fs}) defines the following functional on $\mathcal{P}$
\begin{equation}
\langle f(x)| x^n\rangle= a_n, \label{action}
\end{equation}
for all $n\geq 0.$\\
Particularly, according to (\ref{fs}) and (\ref{action}) we have
\begin{equation}
\langle t^k| x^n\rangle=[n]_q!\delta_{n,k}\quad n,k\geq0, \label{particular}
\end{equation}
where $\delta_{n,k}$ is the Kronecker's symbol.\\
Assume that $f_L(t)=\sum\limits_{k=0}^{\infty}\dfrac{\langle L| x^k\rangle}{[k]_q!}t^k$. Since $\langle f_L(t)| x^n\rangle=\langle L| x^k\rangle$, so $f_L(t)=L$. Hence, it is clear that the map $L\mapsto f_L(t)$ is a vector space isomorphism from  $\mathcal{P}^*$ onto $\mathcal{F}$. Therefore, $\mathcal{F}$ not only can be considered as the algebra of all formal power $q$-series in variable $t$, but also it is the vector space of all linear functionals on $\mathcal{P}$. This follows the fact that each member of $\mathcal{F}$ can be assumed as both a formal power $q$series and a linear functional. $\mathcal{F}$ is called the $q$-umbral algebra and studying its properties is called $q$-umbral calculus.\\
\begin{remark}\label{r1}
For the $q$-exponential function $e_{q}\left(  t\right)$, defined in (\ref{exp}), it can be  easily observed that $\langle e_{q}\left(  yt\right)| x^n\rangle=y^n$ and consequently $$\langle e_{q}\left(  yt\right)| p(x)\rangle=p(y),$$ and $$\langle e_{q}\left(  yt\right)\pm1| p(x)\rangle=p(y)\pm p(0).$$
\end{remark}
\begin{remark}
For $f(t)$ in $\mathcal{F}$ we have
\begin{equation}
f(t)=\sum\limits_{k=0}^{\infty}\frac{\langle f(t)| x^k\rangle}{[k]_q!}t^k,\label{expansionf}
\end{equation}
and for all polynomials $p(x)$ in $\mathcal{P}$ we have
\begin{equation}
p(x)=\sum\limits_{k=0}^{\infty}\frac{\langle t^k| p(x)\rangle}{[k]_q!}x^k.\label{expansionp}
\end{equation}
\end{remark}
\begin{proposition}
For $f(t)$ and $g(t)\in \mathcal{F}$ we have
$$\langle f(t)g(t)| p(x)\rangle=\langle f(t)| g(t)p(x)\rangle.$$
\end{proposition}
\begin{proposition}
For $f(t)$ and $g(t)\in \mathcal{F}$ we have
$$\langle f(t)g(t)| x^n\rangle=\sum\limits_{k=0}^{\infty}
\left[
\begin{array}
[c]{c}%
n\\
k
\end{array}
\right]  _{q}\langle f(t)| x^k\rangle\langle g(t)| x^{n-k}\rangle.$$
\end{proposition}
\begin{proposition}\label{P5}
For $f_1(t), f_2(t), \ldots, f_n(t) \in \mathcal{F}$ we have
\begin{flalign}
\langle f(t)_1f_2(t)\ldots f_k(t)| x^n\rangle& = & \\& & \nonumber \sum\limits_{i_1+i_2+\ldots+i_k=n}
\left[
\begin{array}
[c]{c}%
n\\
i_1,i_2,\ldots,i_k
\end{array}
\right]  _{q}\langle f_1(t)| x^i_1\rangle\langle f_2(t)| x^i_2\rangle\ldots\langle f_k(t)| x^i_k\rangle,
\end{flalign}
where $\left[
\begin{array}
[c]{c}%
n\\
i_1,i_2,\ldots,i_k
\end{array}
\right]  _{q}=\frac{[n]_q!}{[i_1]_q![i_2]_q!\ldots[i_k]_q!}.$
\end{proposition}
We use the notation $t^k$ for the $k$-th $q$-derivative operator, $D_{q}^{k}$, on $\mathcal{P}$ as follows
\begin{equation}
t^kx^n=\Bigg \{\begin{array}{cc}
                      \frac{[n]_q!}{[k]_q!}x^{n-k}, & k\leq n ,\\
           0, & k>n. \\
         \end{array}
\end{equation}
Consequently, using the notation above, each arbitrary function in the form of (\ref{fs}) can be considered as a linear operator on $\mathcal{P}$ defined by
\begin{equation}
f(t)x^n=\sum\limits_{k=0}^{\infty}\left[
\begin{array}
[c]{c}%
n\\
k
\end{array}
\right]  _{q}a_kx^{n-k}.
\label{operator}\end{equation}
Now, consider an arbitrary polynomials $p(x) \in \mathcal{P}$. Then, according to the relation (\ref{expansionp}) for its $k$-th $q$-derivative we have
\begin{equation}
D_q^k p(x)={p}^{(k)}(x)=\sum\limits_{j=k}^{\infty}\frac{\langle t^j| p(x)\rangle}{[j]_q!}[j]_q[j-1]_q\ldots [j-k+1]_q x^{j-k}.
\end{equation}
As the result of the fact above we obtain
\begin{equation}
t^k p(x)=D_q^k p(x)={p}^{(k)}(x),\label{derpol}
\end{equation}
and, also,
\begin{equation}
{p}^{(k)}(0)=\langle t^k| p(x)\rangle=\langle 1| {p}^{(k)}(x)\rangle.\label{t0}
\end{equation}

The immediate conclusion of the relations (\ref{fs}), (\ref{action}) and (\ref{operator}) is that each member of $\mathcal{F}$ plays three roles in the $q$-umbral calculus; a formal power $q$-series, a linear functional and a linear operator.\\
The order of a non-zero power $q$-series $f(t)$ in (\ref{fs}) is denoted by $\textsl{O}(f(t))$ and is defined as the smallest integer $k$ for which the coefficient of $t^k$ is non-zero, that is $a_k\neq0$. A $q$-series $f(t)$ with $\textsl{O}(f(t))=0$ is called invertible and in case that $\textsl{O}(f(t))=1$ it is called a delta $q$-series.
\begin{theorem}\label{thsheffer}
Let $f(t)$ be a delta $q$-series and $g(t)$ be an invertible series. Then there exists a unique sequence $S_{n,q}(x)$ of $q$-polynomials satisfying the following conditions
$$
\langle g(t)f(t)^k|S_{n,q}(x)\rangle=[n]_q!\delta_{n,k},\label{delta}
$$
for all $n,k\geq 0.$
\end{theorem}
\begin{definition}
In Theorem (\ref{thsheffer}), $\{S_{n,q}(x)\}_{n=0}^{\infty}$ is called the $q$-Sheffer sequence for the pair $(g(t),f(t)).$ Moreover, the $q$-Sheffer sequences for $(g(t),t)$ is the $q$-Appell sequence for $g(t).$
\end{definition}
\begin{theorem} Let $A_{n,q}(x)$ be $q$-Appell for $g(t)$. Then\\ \label{thappellexpansion}
\begin{itemize}
\item[a)] (The Expansion Theorem) for any $h(t)$ in $\mathcal{F}$
$$
h(t)=\sum\limits_{k=0}^{\infty}\frac{\langle h(t)| A_{k,q}(x)\rangle}{[k]_q!} g(t)t^k,
$$
\item[b)] (The Polynomial Expansion Theorem) for any $p(x)$ in $\mathcal{P}$ we have
$$
p(x)=\sum\limits_{k=0}^{\infty}\frac{\langle g(t)t^k| p(x)\rangle}{[k]_q! }A_{k,q}(x).
$$
\end{itemize}
\end{theorem}
\begin{theorem} The following facts are equivalent\\
\label{thequivalent}
\begin{itemize}
\item[a)] $A_{n,q}(x)$ is $q$-Appell for $g(t).$
\\
\item[b)]
$tA_{n,q}(x)=[n]_qA_{n-1,q}(x),$ where $tA_{n,q}(x)=D_q(A_{n,q}(x)).$
\\
\item[c)] For all $ y \in \mathbb{C}\quad \frac{1}{g(t)}e_q(tx)=\sum\limits_{k=0}^{\infty}\frac{A_{k,q}(x)}{[k]_q!}t^k.$
\\
\item[d)]
$A_{n,q}(x)=\sum\limits_{k=0}^{\infty}\left[\begin{array}
[c]{c}%
n\\
k
\end{array}
\right]  _{q} \langle g^{-1}(t)| x^{n-k}\rangle x^k.$
\\
\item[e)]
$A_{n,q}(x)=g^{-1}(t)x^n.$\\
\end{itemize}
\end{theorem}
\begin{remark}
Based on different selections for $g(t)$ in part (c) of Theorem (\ref{thequivalent}), we obtain various families of $q$-Appell polynomials. For instance, it is clear from relations (\ref{Bernoulli}), (\ref{Euler}) and (\ref{Genocchi}) that taking $g(t)$ as $\frac{e_{q}\left(  t\right)-1}{t}$, $\frac{e_{q}\left(  t\right)+1}{2}$ or $\frac{e_{q}\left(  t\right)+1}{2t}$, leads to construct the families of $q$-Bernoulli, $q$-Euler or $q$-Genocchi polynomials, respectively.
\end{remark}
\begin{theorem}(The Recurrence Formula for $q$-Appell Sequences)\label{threcurrence}
Suppose that $A_{n,q}(x)$ is $q$-Appell for $g(t)$. Then we have\\
$$
A_{n+1,q}(qx)=\Big[qx-q^n\frac{D_{q,t}g(t)}{g(qt)}\Big]A_{n,q}(x).
$$
\end{theorem}
\begin{proof}
We prove this theorem in the light of the technique which is applied in the proof of Theorem 2 in \cite{Mah4}. Since $A_{n,q}(x)$ is $q$-Appell for $g(t)$ we can write
\begin{equation}
\frac{1}{g(t)}e_q(tqx)=\sum\limits_{n=0}^{\infty}A_{n,q}(qx)\frac{t^n}{[n]_q!}.\label{4}
\end{equation}
Take $\frac{1}{g(t)}=A_q(t)$. According to (\ref{3}), $A_q(t)$ is analytic. So, differentiating equation (\ref{4}) and multiplying  both sides of the obtained equality by t, we get
\begin{equation}
\sum\limits_{n=0}^{\infty}[n]_q A_{n,q}(qx)\frac{t^n}{[n]_q!}=A_q(qt)e_q(tqx)\Big[t\frac{D_qA_q(t)}{A_q(qt)}+tqx\Big],
\end{equation}
so it follows that
\begin{equation}
\sum\limits_{n=0}^{\infty}[n]_q A_{n,q}(qx)\frac{t^n}{[n]_q!}=\sum\limits_{n=0}^{\infty}q^nA_{n,q}(x)\frac{t^n}{[n]_q!}\Big[t\frac{D_qA_q(t)}{A_q(qt)}+tqx\Big].
\end{equation}
This means that
\begin{equation}
\sum\limits_{n=0}^{\infty}[n]_q A_{n,q}(qx)\frac{t^n}{[n]_q!}=\sum\limits_{n=1}^{\infty}\Bigg[q^{n-1}A_{n-1,q}(x)\frac{D_qA_q(t)}{A_q(qt)}+qxA_{n-1,q}(x)\Bigg]\frac{t^n}{[n]_q!},
\end{equation}
which is equivalent to write
\begin{equation}
\sum\limits_{n=0}^{\infty}[n]_q A_{n,q}(qx)\frac{t^n}{[n]_q!}=\sum\limits_{n=1}^{\infty}\Bigg[q^{n-1}\frac{D_qA_q(t)}{A_q(qt)}+qx\Bigg]A_{n-1,q}(x)\frac{t^n}{[n]_q!}.\label{5}
\end{equation}
Comparing both sides of identity(\ref{5}), we have
\begin{equation}
A_{n,q}(qx)=\Bigg[q^{n-1}\frac{D_qA_q(t)}{A_q(qt)}+qx\Bigg]A_{n-1,q}(x),
\end{equation}
whence the result.
\end{proof}

\section{$q$-Umbral perspective of $q$-Genocchi numbers and polynomials, an example of $q$-Appell sequences}
Over the pas decades, many results have been derived using Umbral as well as $q$-Umbral methods for different members of the family of Appell and $q$-Appell polynomials. In this section, we look at the characteristics and properties of $q$-Genocchi numbers and polynomials, as an example of the family of $q$-Appell polynomials, from $q$-umbral point of view. Indeed, it is possible to derive similar results to the obtained results here for the $q$-Bernoulli and $q$-Euler polynomials. The interested readers may see, for instance \cite{KIM1}-\cite{DERE}.\\
\subsection{Various results regarding $q$-Genocchi polynomials}
According to relation(\ref{Genocchi}), the sequence of $q$-Genocchi polynomials $\{G_{n,q}(x)\}_{n=0}^{\infty}$ is $q$-Appell for $g(t)=\frac{e_q(t)+1}{2t}$. Therefore, relation (\ref{delta}) for the sequence of $q$-Genocchi polynomials, $\{G_{n,q}(x)\}$, can be expressed as follows
\begin{equation}
\big\langle \frac{e_q(t)+1}{2t}t^k|G_{n,q}(x)\big\rangle=[n]_q!\delta_{n,k}, \quad n,k\geq 0.\label{deltaGen}
\end{equation}
\begin{remark}
As direct corollaries of Theorems (\ref{thequivalent}) and (\ref{threcurrence}) we have
\begin{itemize}
\item[a)] $tG_{n,q}(x)=D_qG_{n,q}(x)=[n]_qG_{n-1,q}(x),$
\\
\item[b)]
$G_{n,q}(x)=\sum\limits_{k=0}^{\infty}\left[\begin{array}
[c]{c}%
n\\
k
\end{array}
\right]  _{q} \bigg\langle \frac{2t}{e_q(t)+1}\bigg| x^{n-k}\bigg\rangle x^k,$
\\
\item[c)]
$G_{n,q}(x)=\frac{2t}{e_q(t)+1}x^n,$

\item[d)]
$G_{n+1,q}(qx)=\Bigg[qx-q^{n-1}\bigg(\frac{e_q(t)(t-1)+1}{2t^2}\bigg)\Bigg]G_{n,q}(x).$
\\
\end{itemize}
\end{remark}
\begin{proposition}
For  $n \in \mathbb{N}$  we have
$$
G_{0,q}=1, \quad \sum_{k=1}^{n}\left[\begin{array}
[c]{c}%
n+1\\
k+1
\end{array}
\right]  _{q}G_{n-k,q}=-[n+1]_q(1+G_{n,q}).
$$
\end{proposition}
\begin{proof}
According to the relations (\ref{q-der}), (\ref{particular}) and (\ref{deltaGen}) we can write
\begin{align*}\label{THG1}
\bigg \langle \frac{e_q(t)+1}{2t}\bigg| x^n \bigg \rangle=\frac{1}{2[n+1]_q}\bigg\langle \frac{e_q(t)+1}{t}\bigg| tx^{n+1}\bigg\rangle \\
& =\frac{1}{2[n+1]_q}=\frac{1}{2}\int_{0}^{1} x^n d_{q}x.
\end{align*}
Therefore, for an arbitrary polynomial $p(x)\in \mathcal{P}$ we can conclude
\begin{equation}
\bigg \langle \frac{e_q(t)+1}{2t}\bigg| p(x) \bigg \rangle=\frac{1}{2}\bigg(\int_{0}^{1} p(x) d_{q}x+p(0)\bigg).\label{genint}
\end{equation}
Now, from one hand if we take  $p(x)=G_{n,q}(x)$, then we have
\begin{align}
\frac{1}{2}\bigg(\int_{0}^{1} G_{n,q}(x) d_{q}x + G_{n,q}(0)\bigg)=\bigg \langle \frac{e_q(t)+1}{2t}\bigg| G_{n,q}(x) \bigg \rangle \nonumber\\
& =\bigg \langle 1 \bigg|\frac{ e_q(t)+1}{2t}G_{n,q}(x) \bigg \rangle
 = \bigg \langle t^0 \big| x^n \bigg \rangle=[n]_q!\delta_{n,0}.\label{th2eq1}
\end{align}
From another hand, considering the fact that
\begin{equation}\label{IG}
G_{n,q}(x)=\sum_{k=0}^{n}\left[\begin{array}
[c]{c}%
n\\
k
\end{array}
\right]  _{q}G_{n-k,q}x^k,
\end{equation}
we can conclude that
\begin{equation}
\int_{0}^{1} G_{n,q}(x) d_{q}x = \sum_{k=0}^{n}\left[\begin{array}
[c]{c}%
n\\
k
\end{array}
\right]  _{q}G_{n-k,q}\int_{0}^{1}x^k d_{q}x
 = \sum_{k=0}^{n}\left[\begin{array}
[c]{c}%
n\\
k
\end{array}
\right]  _{q}\frac{G_{n-k,q}(x)}{[k+1]_q}.\label{th2eq2}
\end{equation}

Comparing identity (\ref{th2eq1}) with (\ref{th2eq2}), we obtain
\begin{equation}
\int_{0}^{1} G_{n,q}(x) d_{q}x =\sum_{k=0}^{n}\left[\begin{array}
[c]{c}%
n\\
k
\end{array}
\right]  _{q}\frac{G_{n-k,q}(x)}{[k+1]_q}=\Big\{\begin{array}{cc}
                                                  2-G_{0,q}(0) & n=0 \\
                                                  -G_{0,q}(0) & n\neq 0
                                                \end{array},
\end{equation}
whence the result.
\end{proof}
\begin{remark}
According to part(b) of Theorem (\ref{thappellexpansion}), for an arbitrary polynomial $p(x) \in \mathcal{P}$ we can write
\begin{align*}
p(x)=\sum_{k=0}^{\infty}\langle \frac{e_q(t)+1}{2t}t^k| p(x) \rangle\frac{G_{k,q}(x)}{[k]_q!} \\
& =\frac{1}{2}\sum_{k=0}^{\infty}\langle \frac{e_q(t)+1}{t}| t^k p(x) \rangle\frac{G_{k,q}(x)}{[k]_q!}
& =\frac{1}{2}\sum_{k=0}^{\infty}\frac{ G_{k,q}(x)}{[k]_q!}\Big(\int_{0}^{1} t^kp(x) d_{q}x+t^kp(0)\Big).
\end{align*}
\end{remark}
\begin{remark}
We know that
$$
\langle e_q(t)t^k|(x-1)_q^n \rangle=[n]_q!\delta_{n,k}.
$$
Therefore, according to part(b) of Theorem (\ref{thappellexpansion}), for $G_{n,q}(x)$ as a polynomial chosen from $\mathcal{P}$ we can obtain
\begin{align*}
G_{n,q}(x)=\sum_{k=0}^{n}\langle e_q(t)| t^k G_{n,q}(x) \rangle\frac{(x-1)_q^n}{[k]_q!}\\
&=\sum_{k=0}^{n}\left[\begin{array}
[c]{c}%
n\\
k
\end{array}
\right]  _{q}G_{n-k,q}(1)(x-1)_q^n.
\end{align*}
\end{remark}
\begin{proposition}
For  $n \in \mathbb{N}$  we have
\begin{align*}
(x-1)_q^n= \\
& \frac{1}{2}\Bigg(\sum_{k=0}^{n}\sum_{l=0}^{n-k}\left[\begin{array}
[c]{c}%
n\\
k
\end{array}
\right]  _{q}\left[\begin{array}
[c]{c}%
n-k\\
l
\end{array}
\right]  _{q}\frac{1}{[m+1]_q}G_{k,q}(x)(-1)^{n-k-l}q^{\frac{l(l-1)}{2}}+\sum_{k=0}^{n}\left[\begin{array}
[c]{c}%
n\\
k
\end{array}
\right]  _{q}G_{k,q}(x)\Bigg).
\end{align*}
\end{proposition}
\begin{proof}
From the binomial relation(\ref{Binomial}), we obtain
\begin{equation}
(x-1)_q^n=\sum_{l=0}^{n}(-1)^{n-l}q^{\frac{l(l-1)}{2}}x^l.\label{Binomial1}
\end{equation}
Now, taking $k$-th $q$-derivative from both sides of identity(\ref{Binomial1}), we have
\begin{equation}
t^k(x-1)_q^n=\sum{l=k}^{n}\left[\begin{array}
[c]{c}%
n\\
k
\end{array}
\right]  _{q}=\frac{[n]_q!}{[n-k]_q!}(x-1)_q^{n-k}
\end{equation}
According to part(b) of Theorem (\ref{thappellexpansion}), we can write
\begin{align}
(x-1)_q^n=\sum_{k=0}^{n}\frac{1}{[k]!}\Big\langle \frac{e_q(t)+1}{2t}t^k| (x-1)_q^n \Big\rangle G_{n,q}(x)\\ \nonumber
 =\sum_{k=0}^{n}\left[\begin{array}
[c]{c}%
n\\
k
\end{array}
\right]  _{q}G_{n,q}(x)\Big\langle \frac{e_q(t)+1}{2t}| (x-1)_q^{n-k} \Big\rangle \\ \nonumber
= \sum_{k=0}^{n}G_{n,q}(x)\Big( \int_{0}^{1}(x-1)_q^{n-k}d_qx+1\Big)\\ \nonumber
=\frac{1}{2}\Bigg(\sum_{k=0}^{n}\sum_{l=0}^{n-k}\left[\begin{array}
[c]{c}%
n\\
k
\end{array}
\right]  _{q}\left[\begin{array}
[c]{c}%
n-k\\
l
\end{array}
\right]  _{q}\frac{1}{[m+1]_q}G_{k,q}(x)(-1)^{n-k-l}q^{\frac{l(l-1)}{2}}+\sum_{k=0}^{n}\left[\begin{array}
[c]{c}%
n\\
k
\end{array}
\right]  _{q}G_{k,q}(x)\Bigg).
\end{align}
\end{proof}

\begin{theorem}
Let $\mathcal{P}_n=\{p(x)\in \mathcal{P}|deg(p(x))\leq n\}$. Then for an arbitrary $p(x) \in \mathcal{P}_n$ and a constant $c_{n,q}$, we may assume that
$
p(x)=\sum_{i=0}^{n}c_{i,q}G_{i,q}(x).
$
 Then for any constant $k$, the coefficient $c_{k,q}$ is equal to $\frac{1}{[k]_q!}\big\langle \frac{e_q(t)+1}{2t}|p^{(k)}(x)\big\rangle$, and it can be obtained from the following relation
$$
c_{k,q}=\frac{1}{2[k]_q!}\Big(\int_{0}^{1}p^{(k)}(x)d_qx+p^{(k)}(0)\Big),
$$
where $p^{(k)}(x)=D_q^kp(x).$
\end{theorem}
\begin{proof}
For any polynomial $p(x)=\sum_{i=0}^{n}c_{i,q}G_{i,q}(x)$ in $\mathcal{P}_n$, we may write
\begin{equation}
\big\langle \frac{e_q(t)+1}{2t}t^k|p(x)\big\rangle=\sum_{i=0}^{n}c_{i,q}\big\langle \frac{e_q(t)+1}{2t}t^k|G_{i,q}(x)\big\rangle.
\end{equation}
So, according to the relation (\ref{deltaGen}), we obtain
\begin{equation}
=\sum_{i=0}^{n}c_{i,q}[i]_q!\delta_{i,k}=[k]_q!c_{k,q},
\end{equation}
which means that
\begin{equation}
c_{k,q}=\frac{1}{[k]_q!}\big\langle \frac{e_q(t)+1}{2t}t^k|p(x)\big\rangle.
\end{equation}
According to the relation (\ref{derpol}), this is equivalent to write
\begin{equation}\label{ck}
c_{k,q}=\frac{1}{[k]_q!}\big\langle \frac{e_q(t)+1}{2t}|t^k p(x)\big\rangle=\frac{1}{[k]_q!}\big\langle \frac{e_q(t)+1}{2t}|p^{(k)}(x)\big\rangle.
\end{equation}
finally, using the relation (\ref{genint}), we obtain
\begin{equation}
c_{k,q}=\frac{1}{2[k]_q!}\Big(\int_{0}^{1}p^{(k)}(x)d_qx+p^{(k)}(0)\Big).
\end{equation}
\end{proof}
\subsection{Some results regarding $q$-Genocchi polynomials of higher order}
Let $q\in \mathbb{C}, m\in \mathbb{N}$ and $0<|q|<1$. The $q$-Genocchi polynomials $G^{[m]}_{n,q}(x)$ in $x$, of order $m$, in a suitable neighborhood of $t=0$, are defined by means of the following generating function, \cite{Mah1}
\begin{equation}
\Big( \frac{2t}{e_q(t)+1}\Big)^m e_q(tx)=\sum_{n=0}^{\infty}G^{[m]}_{n,q}(x)\frac{t^n}{[n]_q!}.\label{GenOrder}
\end{equation}
In case that $x=0$, $G^{[m]}_{n,q}(0)=G^{[m]}_{n,q}$ is called the $n$-th $q$-Genocchi number of order $m$. \\
From the above definition, it is clear that the class of $q$-Genocchi polynomials, $\{G^{[m]}_{n,q}(x)\}_{n=0}^{\infty}$, of order $m$ is $q$-Appell for $g(t)=\Big( \frac{e_q(t)+1}{2t}\Big)^m$. Thus, according to the relation (\ref{delta}), for the sequence of $q$-Genocchi polynomials, $G^{[m]}_{n,q}(x)$, of order $m$, we can write
\begin{equation}
\Big\langle \Big( \frac{e_q(t)+1}{2t}\Big)^m t^k|G^{[m]}_{n,q}(x)\Big\rangle=[n]_q!\delta_{n,k}, \quad n,k\geq 0.\label{deltaOrderGen}
\end{equation}

\begin{lemma}
For any $n\in\mathbb{N}_{0}$, the following identity holds for the $n$-th $q$-Genocchi number of order $m$
$$
G^{[m]}_{n,q}=\sum\limits_{i_1+i_2+\ldots+i_m=n} \nonumber
\left[
\begin{array}
[c]{c}%
n\\
i_1,i_2,\ldots,i_m
\end{array}
\right]  _{q}G_{i_1,q}G_{i_2,q}\ldots G_{i_m,q}
$$
\end{lemma}
\begin{proof}
From one hand, according to the relation (\ref{GenOrder}), it is obvious that
\begin{equation}
\Big\langle \Big( \frac{2t}{e_q(t)+1}\Big)^m t^k|x^n\Big\rangle=\sum_{k=0}^{\infty}\frac{G^{[m]}_{n,q}}{[k]_q!}\langle t^k| x^n \rangle=G^{[m]}_{n,q}.
\end{equation}
From another hand, according to the Proposition (\ref{P5}), we have
\begin{equation}
G^{[m]}_{n,q}=\sum\limits_{i_1+i_2+\ldots+i_m=n}
\left[
\begin{array}
[c]{c}%
n\\
i_1,i_2,\ldots,i_m
\end{array}
\right]  _{q}\langle \frac{2t}{e_q(t)+1}| x^{i_1}\rangle\langle \frac{2t}{e_q(t)+1}| x^{i_2}\rangle\ldots\langle \frac{2t}{e_q(t)+1}| x^{i_m}\rangle.
\end{equation}
Based on the relations (\ref{Genocchi}) and (\ref{particular})  for each $\langle \frac{2t}{e_q(t)+1}| x^{i_l}\rangle, \quad l \in \{1, 2, \ldots , m \}$ we can write
\begin{equation}
\big\langle \frac{2t}{e_q(t)+1}|x^{i_l}\big\rangle=\sum_{k=0}^{\infty}\frac{G_{i_l,q}}{[k]!}\langle t^k|x^{i_l}\rangle=G_{i_l,q},
\end{equation}
whence the result.
\end{proof}
\begin{theorem}
For any $n\in\mathbb{N}_{0}$, the following identity holds for the $n$-th $q$-Genocchi polynomial of order $m$
$$
G^{[m]}_{n,q}(x)=\sum_{k=0}^{n}\left[
\begin{array}
[c]{c}%
n\\
k
\end{array}
\right]  _{q}\big\langle \frac{e_q(t)+1}{2t}|G^{[m]}_{n-k,q}(x) \big\rangle G_{k,q}(x)
=\frac{1}{2^{m-1}}\sum_{k=0}^{n}\left[
\begin{array}
[c]{c}%
n\\
k
\end{array}
\right]  _{q}G^{[m-1]}_{n-k,q}G_{k,q}(x).
$$
\end{theorem}
\begin{proof}
According to the relation (\ref{GenOrder}), it is clear that
\begin{equation}
G^{[m]}_{n,q}(x)=\sum_{k=0}^{n}\left[
\begin{array}
[c]{c}%
n\\
k
\end{array}
\right]_{q}G^{[m]}_{n-k,q}x^k.\label{xk}
\end{equation}
Therefore, we may assume that $G^{[m]}_{n,q}(x)=\sum_{k=0}^{n}c_{k,q}G_{k,q}(x)$ is a polynomial with degree $n$ in $\mathcal{P}_n$. Since $G^{[m]}_{n,q}(x)$ is a $q$-Appell polynomial, according to part(b) of Theorem (\ref{thequivalent}) for its $k$-th $q$-derivative we can write
\begin{equation}
D_q^k G^{[m]}_{n,q}(x)=[n]_q[n-1]_q \ldots [n-k+1]_q G^{[m]}_{n-k,q}(x)=\frac{[n]_q!}{[n-k]_q!}G^{[m]}_{n-k,q}(x).
\end{equation}
Now, according to the relation (\ref{ck}), we may continue as
\begin{align}\label{ck1}\nonumber
c_{k,q}=  \frac{1}{[k]_q!}\big\langle \frac{e_q(t)+1}{2t}|t^k G^{[m]}_{n,q}(x) \big\rangle \\ & & & &
=\frac{1}{[k]_q!}\big\langle \frac{e_q(t)+1}{2t}|D_q^k G^{[m]}_{n,q}(x) \big\rangle
=\left[
\begin{array}
[c]{c}%
n\\
k
\end{array}
\right]  _{q}\big\langle \frac{e_q(t)+1}{2t}|G^{[m]}_{n-k,q}(x) \big\rangle.
\end{align}
According to part(e) of  Theorem (\ref{thequivalent}), it is clear that the $q$-Appell polynomial $G^{[m]}_{n-k,q}(x)$ is equal to $\Big( \frac{e_q(t)+1}{2t}\Big)^m x^{n-k}$. As the result of this fact and noting to the relation (\ref{t0}), we obtain from the last identity in (\ref{ck1})
\begin{equation}
c_{k,q}=\left[
\begin{array}
[c]{c}%
n\\
k
\end{array}
\right]  _{q}\big\langle t^0|\frac{2t}{e_q(t)+1}\Big( \frac{e_q(t)+1}{2t}\Big)^m x^{n-k} \big\rangle=\frac{1}{2^{m-1}}\left[
\begin{array}
[c]{c}%
n\\
k
\end{array}
\right]  _{q}G^{[m-1]}_{n-k,q},
\end{equation}
whence the result.
\end{proof}
\begin{theorem}\label{GPS}
For any arbitrary polynomial $p(x) \in \mathcal{P}_n$ the following identity holds
$$p(x)=\sum_{k=0}^{n}\Big\langle \Big( \frac{e_q(t)+1}{2t}\Big)^m t^k\big |p(x)\Big\rangle\frac{G^{[m]}_{k,q}(x)}{[k]_q!} .$$
\end{theorem}
\begin{proof}
Assume that $p(x)=\sum_{i=0}^{n}c_{i,q}G^{[m]}_{i,q}(x)$. Therefore, noting to the  relation (\ref{deltaOrderGen}) for the $q$-Appell polynomial $G^{[m]}_{i,q}(x)$, we may conclude that
\begin{equation}
\Big\langle \Big( \frac{e_q(t)+1}{2t}\Big)^m t^k\big |p(x)\Big\rangle=\sum_{i=0}^{n}c_{i,q}\Big\langle \Big( \frac{e_q(t)+1}{2t}\Big)^m t^k\big|G^{[m]}_{i,q}(x)\Big\rangle=\sum_{i=0}^{n}c_{i,q}[i]_q!\delta_{i,k}=c_{k,q}[k]_q!.
\end{equation}
Thus,
\begin{equation}
c_{k,q}=\frac{1}{[k]_q!}\Big\langle \Big( \frac{e_q(t)+1}{2t}\Big)^m t^k\big |p(x)\Big\rangle.
\end{equation}
Substituting $c_{k,q}$  in the summation assumed in the beginning of the proof, leads to obtain the desired result.
\end{proof}
\begin{theorem}\label{THG2}
For any $n\in\mathbb{N}_{0}$ and any $m\in \mathbb{N}$, the $n$-th $q$-Genocchi polynomial can be expressed based on the following relation
\begin{equation*}
\begin{split}
G_{n,q}(x) & =\sum_{k=0}^{m-1}\frac{\left[
\begin{array}
[c]{c}%
m\\
k
\end{array}
\right]_{q}}{2^m[m]_q!\left[
\begin{array}
[c]{c}%
n+m-k\\
m-k
\end{array}
\right]_{q}}\times \\
& \Bigg\{\sum_{i=0}^{m}\left[
\begin{array}
[c]{c}%
m\\
i
\end{array}
\right]_{q}\sum_{l=0}^{n+m-k}\sum\limits_{l_1+l_2+\ldots+l_i=l}
\left[
\begin{array}
[c]{c}%
l\\
l_1,l_2,\ldots,l_i
\end{array}
\right]_{q}\left[
\begin{array}
[c]{c}%
n+m-k\\
l
\end{array}
\right]_{q}G_{n+m-k-l,q}\Bigg\}G^{[m]}_{k,q}(x)\end{split}
\end{equation*}
\begin{equation*}
\begin{split}
 & + \sum_{k=m}^{n}\frac{\left[
\begin{array}
[c]{c}%
n\\
k-m
\end{array}
\right]_{q}}{2^m[k]_q!\left[
\begin{array}
[c]{c}%
k\\
m
\end{array}
\right]_{q}}\times\\ &
\Bigg\{\sum_{i=0}^{m}\sum_{l=0}^{n-k+m}\sum\limits_{l_1+l_2+\ldots+l_i=l}
\left[
\begin{array}
[c]{c}%
l\\
l_1,l_2,\ldots,l_i
\end{array}
\right]_{q}\left[
\begin{array}
[c]{c}%
n+m-k\\
l
\end{array}
\right]_{q}G_{n-k+m-l,q}\Bigg\}G^{[m]}_{k,q}(x)
\end{split}
\end{equation*}
\end{theorem}
\begin{proof}
In Theorem (\ref{GPS}), take $p(x)$ to be the $n$-th $q$-Genocchi polynomial $G_{n,q}(x)$, that is
\begin{equation}\label{I0}
G_{n,q}(x)=\sum_{k=0}^{n}c_{k,q}G^{[m]}_{k,q}(x),
\end{equation}
where
\begin{equation}\label{km}
c_{k,q}=\frac{1}{[k]_q!}\Big\langle \Big( \frac{e_q(t)+1}{2t}\Big)^m t^k\big |G_{n,q}(x)\Big\rangle.
\end{equation}
Then, for $k < m$, we have
\begin{align*}
c_{k,q}=\frac{1}{2^m[k]_q!}\Big\langle \frac{ {(e_q(t)+1)}^m}{t^{m-k}}|G_{n,q}(x)\Big\rangle \\
=\frac{1}{2^m[k]_q!}.\frac{1}{[n+m-k]_q!\ldots [n+1]_q!}\Big\langle {(e_q(t)+1)}^m{\Big(\frac{1}{t}}\Big)^{m-k}|t^{m-k}G_{n+m-k,q}(x)\Big\rangle \\
=\frac{[m]_q!}{2^m[k]_q![m-k]_q!}.\frac{[m-k]_q!}{[n+m-k]_q!\ldots [n+1]_q!}\Big\langle {(e_q(t)+1)}^m|G_{n+m-k,q}(x)\Big\rangle\\
=\frac{\left[
\begin{array}
[c]{c}%
m\\
k
\end{array}
\right]  _{q}}{2^m}.\frac{[m-k]_q!}{[m]_q![n+m-k]_q!\ldots [n+1]_q!}\Big\langle \sum_{i=0}^{m}\left[
\begin{array}
[c]{c}%
m\\
i
\end{array}
\right]_{q}{(e_q(t))}^m|G_{n+m-k,q}(x)\Big\rangle\\
=\frac{\left[
\begin{array}
[c]{c}%
m\\
k
\end{array}
\right]_{q}}{2^m[m]_q!\left[
\begin{array}
[c]{c}%
n+m-k\\
m-k
\end{array}
\right]_{q}}\Big\langle\sum_{i=0}^{m}\left[
\begin{array}
[c]{c}%
m\\
i
\end{array}
\right]_{q}{(e_q(t))}^m|G_{n+m-k,q}(x)\Big\rangle.
\end{align*}
Applying relation (\ref{IG}) to $G_{n+m-k,q}(x)$, we may continue as
\begin{align*}
c_{k,q}=\frac{\left[
\begin{array}
[c]{c}%
m\\
k
\end{array}
\right]_{q}}{2^m[m]_q!\left[
\begin{array}
[c]{c}%
n+m-k\\
m-k
\end{array}
\right]_{q}}\Big\langle\sum_{i=0}^{m}\left[
\begin{array}
[c]{c}%
m\\
i
\end{array}
\right]_{q}{(e_q(t))}^m|\sum_{l=0}^{n+m-k}\left[
\begin{array}
[c]{c}%
n+m-k\\
l
\end{array}
\right]_{q}G_{n+m-k-l,q}x^l\Big\rangle.
\end{align*}
Using Proposition (\ref{P5}) and considering Remark (\ref{r1}), we obtain
\begin{equation}\label{I1}
c_{k,q} =\frac{\left[
\begin{array}
[c]{c}%
m\\
k
\end{array}
\right]_{q}}{2^m[m]_q!\left[
\begin{array}
[c]{c}%
n+m-k\\
m-k
\end{array}
\right]_{q}}\sum_{i=0}^{m}\left[
\begin{array}
[c]{c}%
m\\
i
\end{array}
\right]_{q}\sum_{l=0}^{n+m-k}\sum\limits_{l_1+l_2+\ldots+l_i=l}
\left[
\begin{array}
[c]{c}%
l\\
l_1,l_2,\ldots,l_i
\end{array}
\right]_{q}\left[
\begin{array}
[c]{c}%
n+m-k\\
l
\end{array}
\right]_{q}G_{n+m-k-l,q}.
\end{equation}
Now, assume that $k\geq m$. Then starting from the relation (\ref{km}), we have
\begin{flalign*}
c_{k,q}=\frac{1}{[k]_q!}\langle( \frac{e_q(t)+1}{2t})^m t^k|G_{n,q}(x)\rangle \\
=\frac{1}{2^m[k]_q!}\langle(e_q(t)+1)^m t^{k-m}|G_{n,q}(x)\rangle
=\frac{1}{2^m[k]_q!}\langle (e_q(t)+1)^m |t^{k-m}G_{n,q}(x)\rangle \\
=\frac{1}{2^m[k]_q!}.\frac{1}{[n+m-k]_q!\ldots [n+1]_q!}\Big\langle(e_q(t)+1)^m \big|G_{n-k+m,q}(x) \Big\rangle\\
=\frac{1}{2^m[k]_q!}.\frac{[n]_q![k-m]_q!}{[n-k-m]_q![k-m]_q!}\Big\langle(e_q(t)+1)^m \big|G_{n-k+m,q}(x) \Big\rangle\\
=\frac{[k-m]_q!}{2^m[k]_q!}\left[
\begin{array}
[c]{c}%
n\\
k-m
\end{array}
\right]_{q}\sum_{i=0}^{m}\Big\langle(e_q(t)+1)^i \big|G_{n-k+m,q}(x) \Big\rangle.
\end{flalign*}
Finally, we obtain
\begin{equation}\label{I2}
c_{k,q}=\frac{\left[
\begin{array}
[c]{c}%
n\\
k-m
\end{array}
\right]_{q}}{2^m[k]_q!\left[
\begin{array}
[c]{c}%
k\\
m
\end{array}
\right]_{q}}\sum_{i=0}^{m}\sum_{l=0}^{n-k+m}\sum\limits_{l_1+l_2+\ldots+l_i=l}
\left[
\begin{array}
[c]{c}%
l\\
l_1,l_2,\ldots,l_i
\end{array}
\right]_{q}\left[
\begin{array}
[c]{c}%
n+m-k\\
l
\end{array}
\right]_{q}G_{n-k+m-l,q}.
\end{equation}
Replacing identities (\ref{I1}) and (\ref{I2}) in the assumed sum in (\ref{I0}), completes the proof.
\end{proof}
\begin{remark}
According to the proof of Theorem (\ref{THG2}), for any $n\in\mathbb{N}_{0}$ and any $m\in \mathbb{N}$, the $n$-th $q$-Appell polynomial, $A_{n,q}(x)$, can be expressed based on the following relation
\begin{equation*}
\begin{split}
A_{n,q}(x) & =\sum_{k=0}^{m-1}\frac{\left[
\begin{array}
[c]{c}%
m\\
k
\end{array}
\right]_{q}}{2^m[m]_q!\left[
\begin{array}
[c]{c}%
n+m-k\\
m-k
\end{array}
\right]_{q}}\times \\
& \Bigg\{\sum_{i=0}^{m}\left[
\begin{array}
[c]{c}%
m\\
i
\end{array}
\right]_{q}\sum_{l=0}^{n+m-k}\sum\limits_{l_1+l_2+\ldots+l_i=l}
\left[
\begin{array}
[c]{c}%
l\\
l_1,l_2,\ldots,l_i
\end{array}
\right]_{q}\left[
\begin{array}
[c]{c}%
n+m-k\\
l
\end{array}
\right]_{q}A_{n+m-k-l,q}\Bigg\}G^{[m]}_{k,q}(x)\end{split}
\end{equation*}
\begin{equation*}
\begin{split}
 & + \sum_{k=m}^{n}\frac{\left[
\begin{array}
[c]{c}%
n\\
k-m
\end{array}
\right]_{q}}{2^m[k]_q!\left[
\begin{array}
[c]{c}%
k\\
m
\end{array}
\right]_{q}}\times\\ &
\Bigg\{\sum_{i=0}^{m}\sum_{l=0}^{n-k+m}\sum\limits_{l_1+l_2+\ldots+l_i=l}
\left[
\begin{array}
[c]{c}%
l\\
l_1,l_2,\ldots,l_i
\end{array}
\right]_{q}\left[
\begin{array}
[c]{c}%
n+m-k\\
l
\end{array}
\right]_{q}A_{n-k+m-l,q}\Bigg\}G^{[m]}_{k,q}(x)
\end{split}
\end{equation*}
\end{remark}


\begin{thebibliography}{1}
\bibitem {Andrews}Andrews, George E.; Askey, Richard; Roy, Ranjan Special functions. Encyclopedia of Mathematics and its Applications, 71. Cambridge University Press, Cambridge, 1999. xvi+664 pp.
\bibitem {Kac}Kac, Victor; Cheung, Pokman Quantum calculus. Universitext. Springer-Verlag, New York, 2002. x+112 pp.
\bibitem {calitz1} Carlitz, L. $q$-Bernoulli numbers and polynomials. Duke Math. J. 15, (1948). 987–1000.
\bibitem {calitz2}  Carlitz, L. $q$-Bernoulli and Eulerian numbers. Trans. Amer. Math. Soc. 76, (1954). 332–350.
\bibitem {calitz3} Carlitz, L. Expansions of $q$-Bernoulli numbers. Duke Math. J. 25 1958 355–364.
\bibitem {Kim} Kim, Taekyun On the $q$-extension of Euler and Genocchi numbers. J. Math. Anal. Appl. 326 (2007), no. 2, 1458–1465.
\bibitem{Araci}Araci, Serkan; Açikgöz, Mehmet A note on the values of weighted $q$-Bernstein polynomials and weighted $q$-Genocchi numbers. Adv. Difference Equ. 2015, 2015:30, 9 pp.
\bibitem{Araci1} Araci, S.; Acikgoz, M.; Qi, F.; Jolany, H. A note on the modified $q$-Genocchi numbers and polynomials with weight $(\alpha, \beta)$. Fasc. Math. No. 51 (2013), 21–32.
\bibitem{Kang} Kang, J. Y.; Ryoo, C. S. On symmetric property for $q$-Genocchi polynomials and zeta function. Int. J. Math. Anal. (Ruse) 8 (2014), no. 1-4, 9–16.
\bibitem {Mah}Mahmudov, Nazim I.; Eini Keleshteri, M. On a class of generalized $q$-Bernoulli and $q$-Euler polynomials. Adv. Difference Equ. 2013, 2013:115, 10 pp.
\bibitem {Mah1} Mahmudov, Nazim I.; Keleshteri, Marzieh Eini $q$extensions for the Apostol type polynomials. J. Appl. Math. 2014, Art. ID 868167, 8 pp.
\bibitem {Mah2}Mahmudov, Nazim I. On a class of $q$-Bernoulli and $q$-Euler polynomials. Adv. Difference Equ. 2013, 2013:108, 11 pp.
 \bibitem {Mah3} Mahmudov, N. I.; Momenzadeh, M. On a class of $q$-Bernoulli, $q$-Euler, and $q$-Genocchi polynomials. Abstr. Appl. Anal. 2014, Art. ID 696454, 10 pp.
 \bibitem {luo2}Luo, Qiu-Ming Some results for the $q$-Bernoulli and $q$-Euler polynomials. J. Math. Anal. Appl. 363 (2010), no. 1, 7–18.
 \bibitem {choi1}Srivastava, H. M.; Choi, Junesang Series associated with the zeta and related functions. Kluwer Academic Publishers, Dordrecht, 2001. x+388 pp.
  \bibitem {pinter}Srivastava, H. M.; Pintér, Á. Remarks on some relationships between the Bernoulli and Euler polynomials. Appl. Math. Lett. 17 (2004), no. 4, 375–380.
  \bibitem {sri1}Luo, Qiu-Ming; Srivastava, H. M. $q$-extensions of some relationships between the Bernoulli and Euler polynomials. Taiwanese J. Math. 15 (2011), no. 1, 241–257.
  \bibitem {choi3} Choi, Junesang; Anderson, P. J.; Srivastava, H. M. Carlitz's $q$-Bernoulli and $q$-Euler numbers and polynomials and a class of generalized $q$-Hurwitz zeta functions. Appl. Math. Comput. 215 (2009), no. 3, 1185–1208.
  \bibitem {kim1}Kim, Taekyun Some formulae for the $q$-Bernoulli and Euler polynomials of higher order. J. Math. Anal. Appl. 273 (2002), no. 1, 236–242.
  \bibitem {kim2}Kim, T. $q$-generalized Euler numbers and polynomials. Russ. J. Math. Phys. 13 (2006), no. 3, 293–298.
  \bibitem {kim8}Kim, Taekyun; Rim, Seog-Hoon; Simsek, Yilmaz; Kim, Daeyeoul On the analogs of Bernoulli and Euler numbers, related identities and zeta and $L$-functions. J. Korean Math. Soc. 45 (2008), no. 2, 435–453.
  \bibitem {ozden}Ozden, Hacer; Simsek, Yilmaz A new extension of $q$-Euler numbers and polynomials related to their interpolation functions. Appl. Math. Lett. 21 (2008), no. 9, 934–939.
  \bibitem {manna}O-Yeat Chan, D. Manna, Preprint, oyeat.com/papers/$q$-Bernoulli-20110825.pdf
  \bibitem {ryoo1}Ryoo, C. S.; Seo, J. J.; Kim, T. A note on generalized twisted $q$-Euler numbers and polynomials. J. Comput. Anal. Appl. 10 (2008), no. 4, 483–493.
  \bibitem {simsek1}Simsek, Yilmaz $q$-Analogue of the twisted $L$-series and $q$-twisted Euler numbers, J. Number Theory, 110 (2005), 267-278.
  \bibitem {simsek2} Simsek, Yilmaz Twisted $(h,q)$-Bernoulli numbers and polynomials related to twisted $(h,q)$-zeta function and L-function. J. Math. Anal. Appl. 324 (2006), no. 2, 790–804.
  \bibitem {simsek3}Simsek, Yilmaz Generating functions of the twisted Bernoulli numbers and polynomials associated with their interpolation functions. Adv. Stud. Contemp. Math. (Kyungshang) 16 (2008), no. 2, 251–278.
  \bibitem {sri}Srivastava, H. M.; Kim, T.; Simsek, Y. $q$-Bernoulli numbers and polynomials associated with multiple $q$-zeta functions and basic $L$-series. Russ. J. Math. Phys. 12 (2005), no. 2, 241–268.
  \bibitem {Kurt1}Kurt, Burak A further generalization of the Euler polynomials and on the 2D-Euler polynomials. Proc. Jangjeon Math. Soc. 16 (2013), no. 1, 63–69.
  \bibitem {Kurt2}Kurt, Burak A further generalization of the Euler polynomials and on the 2D-Euler polynomials. Proc. Jangjeon Math. Soc. 15 (2012), no. 4, 389–394.
  \bibitem {Appell}Appell, P. ,Une classe de polyn\^{o}mes, Ann. Sci. \'{E}cole Norm. Sup. (2),9 (1880), 119--144.
  \bibitem {Throne}Thorne, C. J. A property of Appell sets. Amer. Math. Monthly 52, (1945). 191–193.
  \bibitem {Amer}Sheffer, I. M. Note on Appell polynomials. Bull. Amer. Math. Soc. 51, (1945). 739–744.
  \bibitem {Varma} Varma, R. S. On Appell polynomials. Proc. Amer. Math. Soc. 2, (1951). 593–596.
  \bibitem {Shar}Sharma, A.; Chak, A. M. The basic analogue of a class of polynomials. Riv. Mat. Univ. Parma 5 (1954), 325–337.
\bibitem {Alsalam}Al-Salam, Walled A. $q$-Appell polynomials. Ann. Mat. Pura Appl. (4) 77 1967 31–45.
\bibitem {Sri} Srivastava, H. M. Some characterizations of Appell and $q$-Appell polynomials. Ann. Mat. Pura Appl. (4) 130 (1982), 321–329.
\bibitem {Lou} Loureiro, Ana F.; Maroni, Pascal Around $q$-Appell polynomials sequences. Ramanujan J. 26 (2011), no. 3, 311–321.
\bibitem {Kishan}Sharma, Kishan; Jain, Renu Lie theory and $q$-Appell functions. Proc. Natl. Acad. Sci. India Sect. A Phys. Sci. 77 (2007), no. 3, 259–261.
\bibitem {Ernst}Ernst, Thomas Convergence aspects for $q$-Appell functions I. J. Indian Math. Soc. (N.S.) 81 (2014), no. 1-2, 67–77.
\bibitem {Mah4}Mahmudov, Nazim I. Difference equations of $q$-Appell polynomials. Appl. Math. Comput. 245 (2014), 539–543.
\bibitem {KIM1}Kim, Dae San; Kim, Taekyun Umbral calculus associated with Bernoulli polynomials. J. Number Theory 147 (2015), 871–882.
\bibitem {KIM2}Kim, Dae San; Kim, Tae Kyun $q$-Bernoulli polynomials and $q$-umbral calculus. Sci. China Math. 57 (2014), no. 9, 1867–1874.
\bibitem{ARACI} Araci, Serkan; Kong, Xiangxing; Acikgoz, Mehmet; Şen, Erdoğan A new approach to multivariate $q$-Euler polynomials using the umbral calculus. J. Integer Seq. 17 (2014), no. 1, Article 14.1.2, 10 pp.
\bibitem{DERE} Dere, Rahime; Simsek, Yilmaz Genocchi polynomials associated with the Umbral algebra. Appl. Math. Comput. 218 (2011), no. 3, 756–761.
\bibitem{RomanRota}Roman S, Rota G. The umbral calculus. Advances Math. 1978;27:95–188.
\bibitem{Roman}Roman S., Rota G. The umbral calculus. Advances Math. 1978;27:95–188.
\bibitem{Roman1}Roman S., More on the urmbral calculus, with emphasis on the $q$-umbral calculus, J. Math. Anal.
Appl. 107, 222–254 (1985).
\bibitem{Roman2}Roman S., The theory of the urnbrai calculus I, J. Math. Anal. Appl. 87 (1982), 58-115.
\bibitem{IM}Sheffer IM, Some properties of polynomial sets of type zero. Duke Math J. 1939;5(3):590–622.


\end{thebibliography}
\end{document}